\documentclass[11pt]{article}
\usepackage[margin=1in]{geometry}
\usepackage{amsmath}
\usepackage{amsfonts}
\usepackage{color}
\usepackage{latexsym}
\usepackage{tablefootnote}
\usepackage{epsfig}
\usepackage{multirow}
\usepackage{float}
\usepackage{array}
\usepackage{hyperref}
\usepackage{amssymb,amsthm}
\usepackage{algorithm,algorithmic}

\newcommand*{\QEDA}{\hfill\hbox{\vrule width1.0ex height1.0ex}}

\usepackage{makecell,booktabs}
\usepackage[version=4]{mhchem}
\usepackage[strict]{changepage}

\usepackage{amsmath}
\usepackage{amsfonts}
\usepackage{latexsym}
\usepackage{amssymb}

\newtheorem{thm}{Theorem}[section]
\newtheorem{theorem}[thm]{Theorem}

\newtheorem{lemma}[thm]{Lemma}

\newtheorem{proposition}[thm]{Proposition}

\newtheorem{remark}[thm]{Remark}

\newcommand{\beq}{\begin{equation}}
\newcommand{\eeq}{\end{equation}}
\newcommand{\beqa}{\begin{eqnarray}}
\newcommand{\eeqa}{\end{eqnarray}}
\newcommand{\beqas}{\begin{eqnarray*}}
\newcommand{\eeqas}{\end{eqnarray*}}
\newcommand{\bi}{\begin{itemize}}
\newcommand{\ei}{\end{itemize}}

\newcommand{\R}{\mathbb{R}}

\newcommand{\lam}{{\lambda}}

\newcommand{\inner}[2]{\langle #1,#2\rangle}

\newcommand{\argmin}{\mathrm{argmin}\,}

\newcommand{\dom}{\mathrm{dom}\,}

\newcommand{\Argmin}{\mathrm{Argmin}\,}

\newcommand{\tx}{\tilde x}
\newcommand{\ty}{\tilde y}
\newcommand{\tz}{\tilde z}

\begin{document}
	\title{Unifying restart accelerated gradient and proximal bundle methods}
	\date{January 7, 2025}
	\author{
		Jiaming Liang \thanks{Goergen Institute for Data Science and Department of Computer Science, University of Rochester, Rochester, NY 14620 (email: {\tt jiaming.liang@rochester.edu}).
		}
		 }
	\maketitle
	
	\begin{abstract}
		This paper presents a novel restarted version of Nesterov's accelerated gradient method and establishes its optimal iteration-complexity for solving convex smooth composite optimization problems. The proposed restart accelerated gradient method is shown to be a specific instance of the accelerated inexact proximal point framework introduced in \cite{MonteiroSvaiterAcceleration}. Furthermore, this work examines the proximal bundle method within the inexact proximal point framework, demonstrating that it is an instance of the framework. Notably, this paper provides new insights into the underlying algorithmic principle that unifies two seemingly disparate optimization methods, namely, the restart accelerated gradient and the proximal bundle methods.
		\\
		
		\\
		
	\end{abstract}
	
	\section{Introduction}\label{sec:intro}

    Nesterov's accelerated gradient methods \cite{beck2009fast,ag_nesterov83,nesterov2004,nesterov2018lectures} have been widely employed to solve convex smooth composite optimization (CSCO) problems of the form
	\begin{equation}\label{eq:ProbIntro}
	\phi_*:=\min_{x \in \R^n} \left\{\phi(x):=f(x)+h(x)\right\},
	\end{equation}
    where $f$ is typically a convex and $L$-smooth function, and $h$ is a convex and possibly nonsmooth function with a simple proximal mapping, satisfying $\dom h \subset \dom f$.
    Extensive research has been dedicated to developing a theoretical understanding of these methods \cite{ahn2022understanding,bubeck2015geometric,krichene2015accelerated,lessard2016analysis,su2016differential,wibisono2016variational}, extending their scopes \cite{carmon2018accelerated,nonconv_lan16,lan2011primal,jliang2018double,liang2021average,liang2023average,liang2021fista}, and improving their practical performance \cite{beck2009fast,lin2015universal,monteiro2016adaptive,nesterov2015universal,sujanani2024efficient}.
    
    Among many approaches to enhance the convergence of accelerated gradient methods in practice, in particular to suppress the oscillating behavior, the restart technique has shown remarkable improvement in the context of CSCO \cite{alamo2019gradient,alamo2019restart,alamo2022restart,fercoq2019adaptive,giselsson2014monotonicity,monteiro2016adaptive,necoara2019linear,o2015adaptive,su2016differential,sujanani2024efficient}.
    The most natural restart scheme is to restart the accelerated gradient method after a fixed number of iterations, and an optimal fixed restart scheme is presented in \cite{necoara2019linear}. Various adaptive restart schemes have also been explored in the literature.
    Paper \cite{o2015adaptive} proposes a function restart scheme (i.e., it restarts when the function value increases) and a gradient restart scheme (i.e., it restarts when the momentum term and the negative gradient make an obtuse angle).
    However, this paper only provides a heuristic discussion but no non-asymptotic convergence rate.
    Inspired by an ODE interpretation of Nesterov’s accelerated gradient method, \cite{su2016differential} develops speed restart schemes in both continuous and discrete times. In discrete time, the scheme restarts when $\|x_k-x_{k-1}\|<\|x_{k-1}-x_{k-2}\|$.
    The paper presents a convergence analysis in continuous time, while some constants are just shown to exist.
    Based on a restart condition for estimating the strong convexity parameter, \cite{sujanani2024efficient} proposes parameter-free restarted accelerated gradient methods for strongly convex optimization.
 
    Leveraging the A-HPE framework proposed in \cite{MonteiroSvaiterAcceleration}, we develop a novel restart  accelerated composite gradient (ACG) method and establish the same optimal iteration-complexity as the accelerated gradient method for CSCO problems \eqref{eq:ProbIntro}. A-HPE is a generic framework built on an acceleration scheme. Two specific implementations of A-HPE are given in \cite{MonteiroSvaiterAcceleration}, namely, FISTA \cite{beck2009fast} (a variant of  accelerated gradient method) and an accelerated Newton proximal extragradien method, which is the first optimal second-order method.
    Our main contribution is to show restart ACG is another instance of A-HPE and establish the optimal complexity bound of restart ACG.

    Another contribution of the paper is that it demonstrates the modern proximal bundle (MPB) method \cite{liang2021proximal,liang2024unified}, an optimal method for solving convex nonsmooth composite optimization (CNCO) problems, is indeed an instance of the HPE framework \cite{gonccalves2017improved}, which can be understood as the non-accelerated counterpart of A-HPE for CNCO. As a result, MPB can be viewed as a restarted version of the cutting-plane method. 
    Building upon the novel perspectives of restart ACG and MPB as multi-step implementations of A-HPE and HPE, respectively, this paper fconcludes by offering a qualitative analysis to elucidate the superior practical performance of restart ACG and MPB in comparison to their corresponding single-step counterparts, FISTA and the subgradient method.

\subsection{Basic definitions and notation} \label{subsec:DefNot}
    
    
    A proper function $f: \R^n\rightarrow (-\infty,+\infty]$ is $\mu$-strongly convex for some $\mu > 0$ if for every $x, y \in \dom f$ and $\lam \in [0,1]$,
    \[
    f(\lam x+(1-\lam) y)\le \lam f(x)+(1-\lam)f(y) - \frac{\lam(1-\lam) \mu}{2}\|x-y\|^2.
    \]
	For $\varepsilon \ge 0$, the \emph{$\varepsilon$-subdifferential} of $ f $ at $x \in \dom f$ is denoted by
	\[
	    \partial_\varepsilon f (x):=\left\{ s \in\R^n: f(y)\geq f(x)+\left\langle s,y-x\right\rangle -\varepsilon, \forall y\in\R^n\right\}.
	\]
	We denote the subdifferential of $f$ at $x \in \dom f$ by $\partial f (x)$, which is the set  $\partial_0 f(x)$ by definition.
    For a given subgradient
$f'(x) \in \partial f(x)$, we denote the linearization of convex function $f$ at $x$ by $\ell_f(\cdot;x)$, which is defined as
\[
\ell_f(\cdot;x):=f(x)+\inner{ f'(x)}{\cdot-x\rangle}.
\]
	
	
	
	
	

	\section{Restart ACG}	

This section first reviews an ACG variant used in the paper, then presents the restart ACG method, and finally provides the complexity analysis of restart ACG and shows that it is optimal for CSCO.

\subsection{Review of an ACG variant}\label{subsec:ACG}

	In this subsection, we consider
	\begin{equation}\label{eq:prob}
	    \min \{\psi(x):=g(x)+h(x):x\in \R^n \},
	\end{equation}
	where $ g $ is $ \mu $-strongly convex and $ (L+\mu) $-smooth, and $ h $ is as in \eqref{eq:ProbIntro}.
We describe an ACG variant tailored to \eqref{eq:prob} and present some basic results regarding the ACG variant.

\begin{algorithm}[H]
\caption{Accelerated Composite Gradient}\label{alg:ACG}
\begin{algorithmic}
\REQUIRE given initial point $ x_0\in \R^n $, set $A_0=0$, $\tau_0=1/L$, and $ y_0=x_0$

\FOR{$j=0,1,\cdots$}

\STATE {\bf 1.} Compute
		\begin{equation}\label{def:tx}
		    \tau_{j+1}=\tau_j + \frac{\mu a_j}{L}, \quad a_j=\frac{\tau_j + \sqrt{\tau_j^2+4\tau_j A_j}}{2}, \quad A_{j+1} = A_j + a_j, \quad \tx_j=\frac{A_j}{A_{j+1}}y_j + \frac{a_j}{A_{j+1}}x_j;
		\end{equation}

\STATE {\bf 2.} Compute  
		\begin{align}
		    \ty_{j+1} &=\underset{u\in \R^n}\argmin\left\lbrace \ell_g(u;\tx_j) + h(u) + \frac{L+\mu}{2}\|u-\tx_j\|^2\right\rbrace, \label{def:tyj} \\
            y_{j+1} &\in \Argmin\left\lbrace \psi(u): u\in \{y_j, \ty_{j+1}\}\right\rbrace, \label{def:yj} \\
            x_{j+1} &=\frac{(L+\mu)a_j\ty_{j+1} - \frac{A_ja_jL}{A_{j+1}} y_j}{A_{j+1}\mu + 1}. \label{def:xj}
		\end{align}	
\ENDFOR

\end{algorithmic} 
\end{algorithm}

The following three lemmas are standard results for ACG. Therefore, we omit their proofs in this subsection but provide them in the Appendix for completeness.
	
	\begin{lemma}\label{lem:101}
		The following statements hold for every $ j\ge 0 $:
		\begin{itemize}
			\item[(a)] $ \tau_j=(1+\mu A_j)/L $;
			\item[(b)] $ A_{j+1}\tau_j=a_j^2 $.
		\end{itemize}
	\end{lemma}

	\begin{lemma}\label{lem:gamma}
		Define $ \Gamma_0\equiv0 $ and
		\begin{align}
		& \tilde \gamma_j(\cdot):=\ell_g(\cdot;\tx_j) + h(\cdot) + \frac{\mu}{2}\|\cdot-\tx_j\|^2, \label{def:tgamma} \\
		& \gamma_j(\cdot):=\tilde \gamma_j(\ty_{j+1}) + L\inner{\tx_j-\ty_{j+1}}{\cdot-\ty_{j+1}}+\frac{\mu}{2}\|\cdot-\ty_{j+1}\|^2,  \label{def:gamma}\\
		& \Gamma_{j+1}(\cdot):=\frac{A_j \Gamma_j(\cdot) + a_j \gamma_j(\cdot)}{A_{j+1}}. \label{def:Gamma}
		\end{align}        
		Then, the following statements hold for every $j \ge 0$:
		\begin{itemize}
			\item[(a)] $\gamma_j\le \tilde \gamma_j \le \psi$, $\tilde{\gamma}_j(\ty_{j+1}) = \gamma_j(\ty_{j+1})$,
			\begin{equation}\label{eq:equal}
			\underset{u\in \R^n}\min \left\{ \tilde{\gamma}_j(u) + \frac{L}{2}\| u - \tilde{x}_j \|^2 \right\}  = 
			\underset{u\in \R^n}\min \left\{ \gamma_j(u) + \frac{L}{2}\| u - \tilde{x}_j \|^2 \right\},
			\end{equation}
			and these minimization problems have $\ty_{j+1}$ as a unique optimal solution; 
			\item[(b)] $ \gamma_j $ and $ \Gamma_j $ are $ \mu $-strongly convex quadratic functions; 
			\item[(c)] $x_j 
				=\underset{u\in \R^n}\argmin\left\lbrace A_j\Gamma_j(u)+ \|u-x_0\|^2/2 \right\rbrace$.
		\end{itemize}
	\end{lemma}

	\begin{lemma}\label{lem:min}
		For every $j\ge 0$, we have \begin{equation}\label{ineq:induction}
		    A_j \psi(y_j)\le \underset{u\in \R^n}\min\left\lbrace A_j\Gamma_j(u)+\frac12 \|u-x_0\|^2\right\rbrace.
		\end{equation}
	\end{lemma}

    The following lemma is the same as Proposition 1(c) of \cite{monteiro2016adaptive} and hence we omit the proof.
    \begin{lemma}\label{lem:Aj}
		For every $ j\ge 1 $, we have
		\[
		A_j \ge \max \left\{\frac{j^{2}}{4 L}, \frac{1}{L}\left( 1+\frac{\sqrt{\mu}}{2\sqrt{L}}\right) ^{2(j-1)}\right\}.
		\]
	\end{lemma}

\subsection{The algorithm}\label{subsec:algorithm}

Subsection \ref{subsec:ACG} outlines a single-loop ACG method. Designing a restart ACG requires repeatedly invoking Algorithm \ref{alg:ACG} as a subroutine within a double-loop algorithm. This approach aligns naturally with the proximal point method (PPM), which iteratively solves a sequence of proximal subproblems using a recursive subroutine. More precisely, we adopt the A-HPE framework from \cite{MonteiroSvaiterAcceleration} as an inexact PPM. Within each loop of A-HPE, Algorithm~\ref{alg:ACG} is employed to solve a certain proximal subproblem, while between successive loops, an acceleration scheme from A-HPE is applied. Consequently, the proposed restart ACG method (i.e., Algorithm~\ref{alg:restart}) can be described as ``doubly accelerated."

        \begin{algorithm}[H]
\caption{Restart ACG}\label{alg:restart}
\begin{algorithmic}
\REQUIRE given initial point $ w_0\in \dom h $ and stepsize $\lam>0$, set $ z_0=w_0 $ and $ B_0=0 $

\FOR{$k=1,2,\cdots$}

\STATE {\bf 1.} Compute
		\[
		b_{k}=\frac{\lam + \sqrt{\lam^2+4\lam B_{k-1}}}{2}, \quad B_{k} = B_{k-1} + b_{k}, \quad \tilde z_{k}=\frac{B_{k-1}}{B_{k}}w_{k-1} + \frac{b_{k}}{B_{k}}z_{k-1};
		\]

\STATE {\bf 2.} Call Algorithm \ref{alg:ACG} with 
\begin{equation}\label{eq:setup}
    x_0=\tilde z_{k}, \quad \psi= \phi+\frac{1}{2\lam}\|\cdot-\tilde z_{k}\|^2, \quad g=f+\frac{1}{2\lam}\|\cdot-\tilde z_{k}\|^2
\end{equation}
to find a triple $ (\tilde w_{k}, u_{k}, \eta_{k}) $ satisfying
		\begin{align}
		& u_k\in \partial_{ \eta_k} \phi(\tilde w_k),  \label{subdiff-1} \\
		& \| \lam u_k + \tilde w_{k} - \tilde z_{k} \|^2 + 2 \lam \eta_{k} \le 0.9 \|\tilde z_{k}-\tilde w_{k}\|^2; \label{ineq:a-hpe} 
		\end{align}

\STATE {\bf 3.} Compute $z_{k}=z_{k-1}-b_{k} u_{k}$ and $ w_{k}\in \Argmin\left\lbrace \phi(u): u\in \{w_{k-1}, \tilde w_{k}\}\right\rbrace$.
\ENDFOR

\end{algorithmic} 
\end{algorithm}

    From the ``inner loop" perspective, Algorithm~\ref{alg:restart} keeps performing ACG iterations until \eqref{subdiff-1} and \eqref{ineq:a-hpe} are satisfied, and then restarts ACG with the initialization as in \eqref{eq:setup}. 
     From the ``outer loop" perspective, Algorithm~\ref{alg:restart} is an instance of the A-HPE framework of \cite{MonteiroSvaiterAcceleration} for solving \eqref{eq:ProbIntro} with $\lam_k=\lam$ for every $k\ge 1$ and ACG as its subroutine for step 2. The constant $0.9$ in \eqref{ineq:a-hpe} is not critical and can be any arbitrary number within the interval $(0,1)$. With minor modification, such as generalizing $f$ to $\phi$, the results in Section 3 of \cite{MonteiroSvaiterAcceleration} are applicable to this paper. Consequently, Theorem 3.8 of \cite{MonteiroSvaiterAcceleration} also holds. For completeness, we state the theorem below in our context without providing a proof.

    \begin{theorem}\label{thm:outer}
        For every $k\ge 1$, we have
        \[
        \phi(w_k) - \phi_* \le \frac{2d_0^2}{\lam k^2}.
        \]
    \end{theorem}

    We next provide some PPM interpretations of conditions \eqref{subdiff-1} and \eqref{ineq:a-hpe}.
    Each call to ACG in step 2 of Algorithm~\ref{alg:restart} approximately solves the proximal subproblem
    \begin{equation}\label{eq:prox-sub}
        \hat z_k = \underset{u\in \R^n}\argmin\left\{\phi(u)+\frac{1}{2\lam}\|u-\tilde z_{k}\|^2\right\}
    \end{equation}
    such that criteria \eqref{subdiff-1} and \eqref{ineq:a-hpe} are satisfied. They can be equivalently written as 
    \begin{align*}
		& \lam u_k+ \tilde w_k - \tz_k \in \partial_{ \lam \eta_k} \left(\lam \phi(\cdot) + \frac{1}{2}\|\cdot-\tz_k\|^2\right) (\tilde w_k)   \\
		& \| \lam u_k + \tilde w_{k} - \tilde z_{k} \|^2 + 2 \lam \eta_{k} \le 0.9 \|\tilde z_{k}-\tilde w_{k}\|^2.
		\end{align*}
    Alternatively, we can derive a relative solution accuracy guarantee on \eqref{eq:prox-sub}. It follows from \eqref{subdiff-1} that for every $u \in \R^n$,
    \[
    \phi(u) \ge \phi(\tilde w_k) + \inner{u_k}{u-\tilde w_k} - \eta_k,
    \]
    which together with \eqref{ineq:a-hpe} implies that
    \begin{align*}
        0.9 \|\tilde z_{k}-\tilde w_{k}\|^2 &\ge \| \lam u_k + \tilde w_{k} - \tilde z_{k} \|^2 + 2 \lam \eta_{k} \\
         &\ge \| \lam u_k + \tilde w_{k} - \tilde z_{k} \|^2 + 2 \lam [\phi(\tilde w_k) + \inner{u_k}{u-\tilde w_k} - \phi(u)].
    \end{align*}
    Taking $u=\hat z_k$, which is the exact solution to \eqref{eq:prox-sub}, we have
    \begin{align*}
        0.9 \|\tilde z_{k}-\tilde w_{k}\|^2 &\ge \|\lam u_k\|^2 + 2\lam \left[\phi(\tilde w_k) + \frac1{2\lam}\|\tilde w_k-\tz_k\|^2 - \phi(\hat z_k) + \inner{u_k}{\hat z_k -\tilde z_k}\right] \\
        &\ge \|\lam u_k + \hat z_k - \tz_k\|^2 + 2\lam \left[\phi(\tilde w_k) + \frac{1}{2\lam}\|\tilde w_k-\tz_k\|^2 - \phi(\hat z_k) - \frac{1}{2\lam}\|\hat z_k-\tz_k\|^2\right].
    \end{align*}
    Therefore,
    \[
    \phi(\tilde w_k) + \frac{1}{2\lam}\|\tilde w_k-\tz_k\|^2 - \phi(\hat z_k) - \frac{1}{2\lam}\|\hat z_k-\tz_k\|^2 \le \frac{0.9}{2\lam} \|\tilde z_{k}-\tilde w_{k}\|^2,
    \]
    indicating that $\tilde w_k$ is an approximate solution to \eqref{eq:prox-sub} with respect to the relative accuracy given above.
    
    \subsection{Complexity analysis}

    This subsection provides the complexity analysis of Algorithm \ref{alg:restart}. It first establishes the iteration-complexity for ACG to find a triple $(\tilde w_{k}, u_{k}, \eta_{k})$ satisfying \eqref{subdiff-1} and \eqref{ineq:a-hpe}.

    To set the stage, recall the initialization in \eqref{eq:setup}, we note that ACG invoked at step 2 of Algorithm~\ref{alg:restart} solves \eqref{eq:prob} with
    \begin{equation}\label{def:psi}
            \psi(\cdot) = \phi(\cdot) + \frac{1}{2\lam}\|\cdot-x_0\|^2, \quad g(\cdot) = f(\cdot) + \frac{1}{2\lam}\|\cdot-x_0\|^2, \quad \mu = \frac{1}{\lam}.
        \end{equation}
        
	\begin{lemma}\label{lem:translate}
    Define 
    \begin{equation}\label{def:phij}
        \phi_j(\cdot) := \Gamma_j(\cdot) - \frac{1}{2\lam}\|\cdot-x_0\|^2,
    \end{equation}
    where $\Gamma_j$ is as in \eqref{def:Gamma}.
    Then, for every $j\ge 1$, we have
		\begin{equation}\label{incl}
			\hat v_j\in \partial_{\varepsilon_j} \phi(y_j), \quad \|\lam \hat v_j + y_j - x_0\|^2 + 2\lam \varepsilon_j = \|\lam v_j\|^2 + 2\lam[\psi(y_j) - \Gamma_j(x_j)]
		\end{equation}
		where
		\begin{equation}\label{def:vj}
			v_j:=\frac{x_0-x_j}{A_j}, \quad \hat v_j:=\frac{x_0-x_j}{A_j} +\frac{x_0-x_j}{\lam}, \quad \varepsilon_j:=\phi(y_j)-\phi_j(x_j)-\inner{\hat v_j}{y_j-x_j}.
		\end{equation}
	\end{lemma}
	\begin{proof}
		It follows from the optimality condition of \eqref{def:xj} and the definition of $\phi_j$ in \eqref{def:phij} that
		\[
		v_j=\frac{x_0-x_j}{A_j} \stackrel{\eqref{def:xj}}\in \partial \Gamma_j(x_j) \stackrel{\eqref{def:phij}}= \partial \phi_j(x_j) + \frac{1}{\lam} (x_j-x_0),
		\]
		which together with \eqref{def:vj} implies that $\hat v_j\in \partial \phi_j(x_j)$. 
        This inclusion, the fact that $\phi\ge \phi_j$, and the definition of $\varepsilon_j$ in \eqref{def:vj} yields that for every $u\in \R^n$,
		\[
		\phi(u)\ge \phi_j(u)\ge \phi_j(x_j)+\inner{\hat v_j}{u-x_j}=\phi(y_j)+\inner{\hat v_j}{u-y_j}-\varepsilon_j,
		\]
        and hence that the inclusion in \eqref{incl} holds.
        Noting that $\hat v_j = v_j + (x_0-x_j)/\lam$, we have
		\begin{align*}
			&\|\lam \hat v_j + y_j - x_0\|^2 + 2\lam \varepsilon_j\\
			=& \| \lam v_j + y_j -x_j\|^2 + 2\lam \left[\phi(y_j)-\phi_j(x_j)\right]-2\lam \inner{ v_j}{y_j-x_j} - 2\inner{x_0-x_j}{y_j-x_j} \\
			=&\|\lam v_j\|^2 + 2\lam \left[\phi(y_j) + \frac{1}{2\lam}\|y_j-x_0\|^2 - \phi_j(x_j) - \frac{1}{2\lam}\|x_j-x_0\|^2\right].
		\end{align*}
        Finally, the identity in \eqref{incl} follows from the above one and the definitions of $\psi$ and $\phi_j$ in \eqref{def:psi} and \eqref{def:phij}, respectively.
	\end{proof}
	
	\begin{lemma}\label{lem:tech}
		Define
            \begin{equation}\label{def:hat xj}
	\hat x_j :=\argmin\left\{ \Gamma_j(u):u\in \R^n \right\},
	\end{equation}
    where $\Gamma_j$ is as in \eqref{def:Gamma}.
        Assuming that $ A_j \ge 3\lam $, then the following statements hold for every $ j\ge 1$:
		\begin{itemize}
			\item[a)] 
			\begin{equation}\label{ineq:mj}
			\psi (y_j)- \Gamma_j( \hat x_j) \le  \frac1{2A_j} \|\hat x_j-x_0\|^2 \le \frac{1}{A_j-2\lam}\|y_j-x_0\|^2;
			\end{equation}
			\item[b)]
			\begin{equation}\label{ineq:vj}
			    \|v_j\|\le \frac{3\|y_j-x_0\|}{2A_j}.
			\end{equation}
		\end{itemize}
	\end{lemma}
	\begin{proof}
		a)
		Using Lemma \ref{lem:min}, we have
		\[
		\psi (y_j)
		\stackrel{\eqref{ineq:induction}}\le \underset{u\in \R^n}\min \left \{ \Gamma_j (u) + \frac1{2A_j} \|u-x_0\|^2 \right\} \le  \Gamma_j( \hat x_j) +  \frac1{2A_j} \|\hat x_j-x_0\|^2,
		\]
		and hence the first inequality in \eqref{ineq:mj} holds.
        Since $\Gamma_j$ is as in \eqref{def:Gamma}, it follows from Lemma \ref{lem:gamma}(b) and the fact that $\mu=1/\lam$ (see \eqref{def:psi}) that $\Gamma_j$ is $\lam^{-1}$-strongly convex.
		This observation, the above inequality, and the definition of $ \hat x_j $ in \eqref{def:hat xj} thus imply that for every $u \in \R^n$,
		\[
		\psi (y_j)  - \frac1{2A_j} \|\hat x_j-x_0\|^2 \le \Gamma_j( \hat x_j)  \le \Gamma_j(u)- \frac1{2\lam}\|u-\hat x_j\|^2.
		\]
		Taking $u = y_j$ in the above inequality and using the fact that $ \Gamma_j\le \psi $, we obtain
		\[
		\|y_j-\hat x_j\|^2 \le \frac{\lam}{A_j} \|\hat x_j-x_0\|^2.
		\]
		Using the above inequality, the triangle inequality, and the fact that $(a+b)^2\le 2(a^2+b^2)$, we have
		\[
		\|\hat x_j-x_0\|^2 \le 2(\|\hat x_j-y_j\|^2+\|y_j-x_0\|^2)
		\le \frac{2\lam}{A_j}\|\hat x_j-x_0\|^2 + 2\|y_j-x_0\|^2,
		\]
		and hence the second inequality in \eqref{ineq:mj} follows.
		
		b) 	It follows from Lemma \ref{lem:min} and the fact that $ \Gamma_j $ is $\lam^{-1} $-strongly convex that 
		\begin{align*}
		\psi(y_j)+\frac12 \left( \frac{1}{\lam} + \frac{1}{A_j}\right) \|u- x_j\|^2 &\le \underset{u\in \R^n}\min\left\lbrace  \Gamma_j(u)+ \frac{1}{2A_j} \|u-x_0\|^2\right\rbrace +\frac12 \left( \frac{1}{\lam} + \frac{1}{A_j}\right) \|u- x_j\|^2\\
		&\le  \Gamma_j(u)+ \frac{1}{2A_j} \|u-x_0\|^2.
		\end{align*}
		Taking $ u=y_j $ in the above inequality and using the fact $ \Gamma_j\le \psi $, we have
		\[
		\frac12 \left( \frac{1}{\lam} + \frac{1}{A_j}\right) \|y_j- x_j\|^2 \le  \Gamma_j(y_j) - \psi(y_j)+ \frac{1}{2A_j} \|y_j-x_0\|^2 \le \frac{1}{2A_j} \|y_j-x_0\|^2,
		\]
		and hence 
		\[
		\|y_j- x_j\|^2 \le \frac{\lam}{\lam + A_j}\|y_j-x_0\|^2 \le \frac14 \|y_j-x_0\|^2
		\]
		where the second inequality is due to $ A_j\ge 3\lam $.	
		 Finally, \eqref{ineq:vj} immediately follows from the definition of $ v_j $ in \eqref{def:vj}, the triangle inequality, and the above inequality. 
	\end{proof}

	\begin{proposition}\label{prop:inner}
		The number of iterations performed by ACG to find a triple $ (\tilde w_k,u_k,\eta_k) $ satisfying \eqref{subdiff-1} and \eqref{ineq:a-hpe} is at most
		\begin{equation}\label{bound}
		\min\left\lbrace 2\sqrt{6\lam L}, \left(\frac12+\sqrt{\lam L}\right)\ln(6\lam L)\right\rbrace.
		\end{equation}
	\end{proposition}
	
	\begin{proof}
        It is easy to verify that \eqref{bound} and Lemma \ref{lem:Aj} imply that $ A_j\ge 6\lam $.
        Consider the first $j$ such that $A_j\ge 6\lam$, then we prove that
        \begin{equation}\label{eq:assign}
            u_k = \hat v_j, \quad \eta_k = \varepsilon_j, \quad \tilde w_k = y_j
        \end{equation}
        satisfy \eqref{subdiff-1} and \eqref{ineq:a-hpe}.
	First, it follows from Lemma \ref{lem:translate} that the inclusion in \eqref{incl} is equivalent to \eqref{subdiff-1} with the assignment \eqref{eq:assign}.
       Using the identity in \eqref{incl} and Lemma \ref{lem:tech} we have
		\begin{align*}
		\|\lam \hat v_j + y_j - x_0\|^2 + 2\lam \varepsilon_j 
		&= \|\lam v_j\|^2 + 2\lam[\psi(y_j) - \Gamma_j(x_j)]\\
		& \stackrel{\eqref{ineq:mj},\eqref{ineq:vj}}\le \frac{9\lam^2\|y_j-x_0\|^2}{4A_j^2} + \frac{2\lam}{A_j-2\lam}\|y_j-x_0\|^2 \\
		& \le \left( \frac{1}{16}+\frac12\right) \|y_j-x_0\|^2,
		\end{align*}	
        where the last inequality is due to the fact that $A_j \ge 6\lam$.
        Hence, \eqref{ineq:a-hpe} also holds in view of \eqref{eq:assign}.
	\end{proof}

Now we are ready to present the main result of the paper.

\begin{theorem}\label{thm:main}
    Given $\bar \varepsilon>0$, assuming that $\lam$ satisfies $1/L \le \lam \le d_0^2/\bar \varepsilon$, then the total iteration-complexity of Algorithm \ref{alg:restart} to find a $\bar \varepsilon$-solution to \eqref{eq:ProbIntro} is ${\cal O}(\sqrt{L} d_0/\sqrt{\bar \varepsilon})$.
\end{theorem}

\begin{proof}
    It follows from Theorem \ref{thm:outer} that to find a $\bar \varepsilon$-solution, the number of calls to ACG is at most $\sqrt{2}d_0/\sqrt{\lam \bar \varepsilon}$. Therefore, the conclusion of the theorem immediately follows from Proposition~\ref{prop:inner} and the assumption on $\lam$.
\end{proof}



    

\section{Connections between restart ACG and MPB}

Inspired by the interpretation of restart ACG as an instance of A-HPE, we revisit the MPB method and show that it is indeed an instance of the HPE framework for CNCO.
To begin with, we present below the HPE framework, adapted from Framework 1 of \cite{gonccalves2017improved}, for solving \eqref{eq:ProbIntro} where $f$ is $M$-Lipschitz continuous instead of being smooth and $h$ is as in \eqref{eq:ProbIntro}.

\begin{algorithm}[H]
\caption{HPE framework}\label{alg:HPE}
\begin{algorithmic}
\REQUIRE given initial point $ w_0\in \dom h $, stepsize $\lam>0$, and tolerance $\delta>0$

\FOR{$k=1,2,\cdots$}

\STATE {\bf 1.} Find a triple $ (\tilde w_{k}, u_{k}, \eta_{k}) $ satisfying
		\begin{align}
		& u_k\in \partial_{ \eta_k} \phi(\tilde w_k),  \label{subdiff} \\
		& \| \lam u_k + \tilde w_{k} - w_{k-1} \|^2 + 2 \lam \eta_{k} \le 2\lam \delta; \label{ineq:hpe} 
		\end{align}

\STATE {\bf 2.} Compute $w_{k}=w_{k-1}-\lam u_{k}$.
\ENDFOR

\end{algorithmic} 
\end{algorithm}

\subsection{MPB as an instance of HPE}

This subsection shows that the MPB method is an instance of the HPE framework. We begin with a brief review of MPB.
A key distinction of MPB from classical PB methods \cite{lemarechal1975extension,lemarechal1978nonsmooth,mifflin1982modification,wolfe1975method} lies in its incorporation of the PPM. MPB approximately solves a sequence of proximal subproblem of the form
\begin{equation}\label{eq:phi-lam}
    \min_{u \in \R^n} \left\{\psi(u) :=\phi(u)+\frac{1}{2 \lambda}\left\|u-w_{k-1}\right\|^2\right\}.
\end{equation}
Letting $x_0=w_{k-1}$ be the initial point of the subroutine for solving \eqref{eq:phi-lam}, MPB iteratively solves 
\begin{equation}\label{eq:PB-xj}
    x_j=\underset{u \in \R^n}{\argmin}\left\{\Gamma_j(u) + h(u) +\frac{1}{2 \lam}\|u-x_0\|^2\right\},
\end{equation}
where $\Gamma_j$ is a bundle model underneath $f$. Details about various models and a unifying framework underlying them are discussed in \cite{liang2024unified}. MPB keeps refining $\Gamma_j$ and solving $x_j$ through \eqref{eq:PB-xj}, until a criterion $t_j =\psi(\tx_j) - m_j \le \delta$ is met, where
\begin{equation}\label{def:tj}
	m_j = \Gamma_j(x_j) + h(x_j) +\frac{1}{2 \lam}\|x_j-x_0\|^2, \quad \tx_j \in \Argmin \{\psi(u): u \in\{x_0,x_1, \ldots, x_j\}\}.
	\end{equation}
As explained in \cite{liang2024primal}, the criterion $t_j \le \delta$ indicates that a primal-dual solution to \eqref{eq:phi-lam} with primal-dual gap bounded by $\delta$ is obtained. It also implies that $\tx_j$ is a $\delta$-solution to \eqref{eq:phi-lam} (see also \cite{liang2021proximal}).
Once the condition $t_j \le \delta$ is met, MPB updates the prox center to $w_k=x_j$, resets the bundle model $\Gamma_j$ from scratch, and proceeds to solve \eqref{eq:phi-lam} with $w_{k-1}$ replaced by $w_k$.

\begin{lemma}\label{lem:bundle}
	Given $x_0=w_{k-1}$, the MPB method is an instance of the HPE framework with 
    \begin{equation}\label{eq:equiv}
        w_k =x_j, \quad \tilde w_k = \tx_j, \quad u_k = \frac{x_0 - x_j}{\lam}, \quad \eta_k = \phi(\tx_j) - (\Gamma_j+h)(x_j)+ \frac{1}{\lam}\inner{x_0-x_j}{x_j-\tx_j},
    \end{equation}
    where $j$ is the first iteration index such that the condition $t_j\le \delta$ is met.
\end{lemma}
\begin{proof}
Let $\varepsilon_j = \phi(\tx_j) - (\Gamma_j+h)(x_j)+ \inner{x_0-x_j}{x_j-\tx_j}/\lam$, we first prove that
\begin{equation}\label{result}
	\frac{ x_0 - x_j}{\lam} \in \partial_{\varepsilon_j} \phi(\tx_j), \quad \|x_j-\tx_j\|^2 + 2\lam \varepsilon_j = 2\lam t_j.
	\end{equation}
	It follows from the optimality condition of \eqref{eq:PB-xj} that
	\[
	\frac{ x_0 - x_j}{\lam}\in \partial (\Gamma_j+h)(x_j).
	\]
    This inclusion and the assumption that $\Gamma_j \le f$ imply that for every $ u\in \R^n $,
	\[
	\phi(u)\ge (\Gamma_j+h)(u)\ge (\Gamma_j+h)(x_j)+\frac{1}{\lam} \inner{ x_0 -x_j}{u-x_j}.
	\]
	Using the definition of $\varepsilon_j$, we thus have for every $ u\in \R^n $,
	\[
	\phi(u)\ge \phi(\tx_j)+\frac{1}{\lam} \inner{ x_0 -x_j}{u-\tx_j} - \varepsilon_j,
	\]
	and hence the inclusion in \eqref{result} follows.
	Next, we show the identity in \eqref{result}.
	Recalling that 
	\[
	2\lam t_j= 2\lam [\psi(\tx_j) - m_j] \stackrel{\eqref{def:tj}}= 2\lam \phi(\tx_j) + \|\tx_j- x_0\|^2 - 2\lam (\Gamma_j+h)(x_j) - \|x_j- x_0\|^2,
	\]
	after simple algebraic manipulation, we obtain
	\[
	\|x_j-\tx_j\|^2 + 2\lam\varepsilon_j = \|x_j-\tx_j\|^2 + 2\inner{ x_0 - x_j}{x_j-\tx_j} + 2\lam[\phi(\tx_j)-(\Gamma_j+h)(x_j)]=2\lam t_j.
	\]
    In view of \eqref{eq:equiv}, it is easy to verify that \eqref{subdiff} is equivalent to the inclusion in \eqref{result}. Moreover, it follows from \eqref{eq:equiv} and the identity in \eqref{result} that $\| w_k - \tilde w_{k}\|^2 + 2 \lam \eta_{k}=2\lam t_j$, which together with step 2 of Algorithm \ref{alg:HPE} implies that $\| \lam u_k + \tilde w_{k} - w_{k-1} \|^2 + 2 \lam \eta_{k} = 2\lam t_j$. Finally, MPB satisfies \eqref{ineq:hpe} since it terminates solving \eqref{eq:phi-lam} once the condition $t_j \le \delta$ is met.
\end{proof}

\subsection{Restart schemes via PPM}

A standard scheme of updating the bundle model $\Gamma_j$ in \eqref{eq:PB-xj} is a cutting-plane scheme, the approximate solutions $x_j$ and $\tx_j$ to \eqref{eq:phi-lam} are obtained via the cutting-plane method. Similar to restart ACG, MPB is also a double-loop algorithm and can be viewed as a restarted version of the cutting-plane method from the ``inner loop" perspective.

On the other hand, from the ``outer loop" perspective, restart ACG (resp., MPB) is an instance of A-HPE (resp., HPE) employing a multi-step subroutine for solving the proximal subproblem \eqref{eq:prox-sub} (resp., \eqref{eq:phi-lam}).
As illustrated by Algorithm 1 of \cite{MonteiroSvaiterAcceleration}, a triple $(\tilde w_k, u_k,\eta_k)$ satisfying \eqref{subdiff-1} and \eqref{ineq:a-hpe} can be obtained via one step (i.e., a proximal mapping of $h$) given the stepsize $\lam\approx 1/L$ is small enough. 
It is noted at the end of Section 5 of \cite{MonteiroSvaiterAcceleration} that its Algorithm 1 is equivalent to the well-known FISTA.
In contrast, Algorithm \ref{alg:restart} of this paper admits relatively large stepsize, i.e., $1/L \le \lam \le d_0^2/\bar \varepsilon$ (see Theorem \ref{thm:main}), and results in a multi-step subroutine, namely, Algorithm~\ref{alg:ACG}, for solving \eqref{eq:prox-sub}.
A common feature between Algorithm~\ref{alg:restart} and FISTA is that they both share the optimal complexity for CSCO problems \eqref{eq:ProbIntro}, namely, ${\cal O}(\sqrt{L} d_0/\sqrt{\bar \varepsilon})$ as in Theorem~\ref{thm:main}.

A similar comparison can be drawn between the MPB method and the subgradient method. MPB allows relatively large stepsize, i.e., $\bar \varepsilon/M^2 \le \lam \le d_0^2/\bar \varepsilon$, while the subgradient method only takes small stepsize $ \lam = \bar \varepsilon/M^2$. As a result, MPB solves the proximal subproblem \eqref{eq:phi-lam} via the cutting-plane method as in \eqref{eq:PB-xj}, while the subgradient method always performs only one iteration to solve \eqref{eq:phi-lam}. For CNCO problems, MPB and the subgradient method both have optimal complexity bound ${\cal O}(M^2 d_0^2/\bar \varepsilon^2)$, however, MPB substantially outperforms the subgradient method in practice.

In summary, the relationship between FISTA (specifically, Algorithm 1 of \cite{MonteiroSvaiterAcceleration}) and the restart ACG method is analogous to the relationship between the subgradient method and MPB. It is thus understandable that, while restart ACG and FISTA share the same optimal iteration-complexity, the restarted version demonstrates superior performance compared with the latter one. This aligns with the general observation that, in the context of both the A-HPE and HPE frameworks, multi-step implementations consistently outperform their single-step counterparts.

\section{Concluding remarks}

This paper proposes a novel restarted version of accelerated gradient method, i.e., restart ACG, and establishes it optimal iteration-complexity for solving CSCO. It also demonstrates that the MPB method as an instance of the HPE framework, revealing interesting connections between two seemingly distinct optimal methods, restart ACG and MPB.

Several related questions merit future investigation. It is interesting to develop a restart accelerated gradient method with optimal complexity under the strong convexity assumption.
If the strong convexity $\mu$ is known, 
it suffices to study the convergence analysis of A-HPE where $h$ is $\mu$-strongly convex.
However, in the absence of prior knowledge about $\mu$, the focus shifts to designing $\mu$-universal methods based on improved analysis of A-HPE, utilizing possible techniques from recent works \cite{guigues2024universal,sujanani2024efficient}.
    
	\bibliographystyle{plain}
	\bibliography{ref}
	
	\appendix

\section{Deferred proofs}

\noindent
\textbf{Proof of Lemma \ref{lem:101}}:
            (a) This statement immediately follows from the recessions of $\tau_j$ and $A_j$ in \eqref{def:tx} and the facts that $\tau_0=1/L$ and $A_0=0$.

            (b) It is easy to verify that $a_j$ in \eqref{def:tx} is the root of equation $a_j^2 - \tau_j a_j - \tau_j A_j =0$, which is equivalent to statement b) in view of the third identity in \eqref{def:tx}.
 \QEDA           

\noindent
\textbf{Proof of Lemma \ref{lem:gamma}}:
		(a) It follows from \eqref{def:tyj} and definition of $\tilde \gamma_j$ in \eqref{def:tgamma} that
        \begin{equation}\label{eq:ty1}
            \ty_{j+1}=\underset{u\in \R^n}\argmin\left\lbrace \tilde \gamma_j(u) + \frac{L}{2}\|u-\tx_j\|^2\right\rbrace.
        \end{equation}
Since $\tilde \gamma_j$ is $\mu$-strongly convex, \eqref{eq:ty1} implies that
\[
\tilde \gamma_j(u) + \frac{L}{2}\|u-\tx_j\|^2 \ge \tilde \gamma_j(\ty_{j+1}) + \frac{L}{2}\|\ty_{j+1}-\tx_j\|^2 + \frac{\mu+L}{2}\|u-\ty_{j+1}\|^2.
\]
Hence, using the definition of $\gamma_j$ in \eqref{def:gamma} and rearranging the terms, we have $\gamma_j\le \tilde \gamma_j$.
Using the definitions of $\psi$ and $\tilde \gamma_j$ in \eqref{eq:prob} and \eqref{def:tgamma}, respectively, and the assumption that $g$ is $\mu$-strongly convex, we obtain $\tilde \gamma_j \le \psi$, and thus prove the inequalities in (a). 
  By the definition of $\gamma_j$ in \eqref{def:gamma}, it is easy to verify that $\tilde{\gamma}_j(\ty_{j+1}) = \gamma_j(\ty_{j+1})$ and     
\begin{equation}\label{eq:ty2}
    \ty_{j+1}=\underset{u\in \R^n}\argmin\left\lbrace \gamma_j(u) + \frac{L}{2}\|u-\tx_j\|^2\right\rbrace.
\end{equation}
Finally, \eqref{eq:equal} is an immediate consequence of \eqref{eq:ty1}, \eqref{eq:ty2} and $\tilde{\gamma}_j(\ty_{j+1}) = \gamma_j(\ty_{j+1})$.
        
		(b) It clearly follows from \eqref{def:gamma} that $\gamma_j$ is $ \mu $-strongly convex quadratic. Moreover, it follows from \eqref{def:Gamma} and the fact that $\Gamma_0\equiv 0$ that $\Gamma_j$ is also $\mu $-strongly convex quadratic.
					
		(c) It follows from \eqref{def:xj}, the definition of $\tx_j$ in \eqref{def:tx}, and Lemma \ref{lem:101}(b) that
        \[
        (A_j \mu +1) x_{j+1} = (L+\mu) a_j\ty_{j+1} - L (a_j \tx_j -\tau_j x_j),
        \]
        which together with Lemma \ref{lem:101}(a) and the third identity in \eqref{def:tx} implies that
        \[
        a_j [L(\tx_j-\ty_{j+1})+\mu(x_{j+1}-\ty_{j+1})] + (A_j \mu +1)(x_{j+1}-x_j) =0.
        \]
        In view of the definition of $\gamma_j$ in \eqref{def:gamma}, the above identity is equivalent to
        \begin{equation}\label{eq:opt}
            a_j \nabla \gamma_j(x_{j+1}) + (A_j \mu +1)(x_{j+1}-x_j) =0.
        \end{equation}
        Hence,
        \[
        x_{j+1} = \underset{u\in \R^n}\argmin\left\lbrace a_j\gamma_j(u) + (A_j\mu+1)\|u-x_j\|^2/2 \right\rbrace.
        \]
        It follows from the definition of $\gamma_j$ in \eqref{def:gamma} that
		\[
		\nabla \gamma_j(x_{j+1}) - \nabla \gamma_j(x_j) = \mu(x_{j+1}-x_j),
		\]
        which together with \eqref{eq:opt} implies that
		\begin{equation}\label{eq:relation}
		    a_j\nabla \gamma_j(x_j)+(A_{j+1}\mu+1)(x_{j+1}-x_j)=0.
		\end{equation}
		It follows from statement b) and \eqref{def:Gamma} that 
		\begin{align*}
			A_{j+1}\nabla \Gamma_{j+1}(x_{j+1}) & = A_{j+1} \nabla \Gamma_{j+1}(x_j) + A_{j+1}\mu(x_{j+1}-x_j)\\
			&\stackrel{\eqref{def:Gamma}}=A_j\nabla \Gamma_j(x_j) + a_j \nabla \gamma_j(x_j) + A_{j+1}\mu(x_{j+1}-x_j)\\
			&\stackrel{\eqref{eq:relation}}=A_j\nabla \Gamma_j(x_j) - x_{j+1}+x_j
		\end{align*}
        where the last identity is due to \eqref{eq:relation}.
		Therefore, for every $\ge 0$,
		\[
		A_{j+1}\nabla \Gamma_{j+1}(x_{j+1}) + x_{j+1}-x_0 = A_j\nabla \Gamma_j(x_j)+x_j-x_0 = A_0\nabla \Gamma_0(x_0)+x_0-x_0 = 0.
		\]
        Therefore, statement c) immediately follows.
\QEDA

    \noindent
\textbf{Proof of Lemma \ref{lem:min}}:
		Proof by induction. Since $A_0=0$, the case $j=0$ is trivial. Assume that the claim is true for some $j\ge 0$. Using \eqref{def:Gamma}, Lemma \ref{lem:gamma}(b), and the induction hypothesis that
        \begin{align*}
		 A_{j+1}\Gamma_{j+1}(u) + \frac{1}{2}\|u-x_0\|^2
		& \stackrel{\eqref{def:Gamma}}\ge A_j\Gamma_j(u) + a_j\gamma_j(u) + \frac{1}{2}\|u-x_0\|^2 \\
        & \ge \underset{u\in \R^n}\min\left\lbrace A_j\Gamma_j(u) + \frac{1}{2}\|u-x_0\|^2\right\rbrace + \frac{A_j\mu + 1}{2}\|u- x_j\|^2 + a_j\gamma_j(u) \\
        & \stackrel{\eqref{ineq:induction}}\ge A_j\psi(y_j) + \frac{A_j\mu + 1}{2}\|u- x_j\|^2 + a_j\gamma_j(u).
		\end{align*}
        It follows from the fact that $\gamma_j \le \tilde \gamma_j \le \psi$ (see Lemma \ref{lem:gamma}(a)) and the definition of $\tx_j$ in \eqref{def:tx} that
        \begin{align*}
		 A_{j+1}\Gamma_{j+1}(u) + \frac{1}{2}\|u-x_0\|^2
		& \ge A_j\gamma_j(y_j) + a_j\gamma_j(u) + \frac{A_j\mu + 1}{2}\|u- x_j\|^2 \\
        & \ge A_{j+1}\gamma_j(\tilde u) + \frac{A_j\mu + 1}{2} \frac{A_{j+1}^2}{a_j^2}\|\tilde u- \tilde x_j\|^2,
		\end{align*}
        where $\tilde u= (A_j y_j + a_j u)/A_{j+1}$ and the second inequality is due to the convexity of $\gamma_j$.
        Minimizing both sides of the above inequality over $\R^n$ and using Lemma \ref{lem:101}(a)-(b), we obtain
		\begin{align*}
		&\underset{u\in \R^n}\min\left\lbrace A_{j+1}\Gamma_{j+1}(u) + \frac{1}{2}\|u-x_0\|^2\right\rbrace 
		 \ge \underset{\tilde u\in \R^n}\min\left\lbrace A_{j+1}\gamma_j(\tilde u) + \frac{A_j\mu + 1}{2} \frac{A_{j+1}^2}{a_j^2}\|\tilde u- \tilde x_j\|^2 \right\rbrace \\
		=& A_{j+1}\underset{\tilde u\in \R^n}\min\left\lbrace \gamma_j(\tilde u) + \frac{L}{2} \|\tilde u- \tilde x_j\|^2 \right\rbrace
		 \stackrel{\eqref{eq:equal}}= A_{j+1} \underset{\tilde u\in \R^n}\min\left\lbrace \tilde \gamma_j(\tilde u) + \frac{L}{2} \|\tilde u- \tilde x_j\|^2 \right\rbrace \\
	 =& A_{j+1} \left(\tilde \gamma_j(\ty_{j+1}) + \frac{L}{2} \|\ty_{j+1}- \tilde x_j\|^2 \right)
		\ge A_{j+1}\psi(\ty_{j+1}),
		\end{align*}
        where the last inequality follows from the definition of $\tilde \gamma_j$ in \eqref{def:tgamma} and the assumption that $g$ is $ (L+\mu) $-smooth.
        Finally, it follows from the definition of $y_{j+1}$ in \eqref{def:yj} that \eqref{ineq:induction} holds for $j+1$. Therefore, we complete the proof by induction.
\QEDA

\end{document}